\documentclass[11pt]{article}

\usepackage[english]{babel}
\usepackage[utf8]{inputenc}
\usepackage{amsmath}
\usepackage{amssymb}
\usepackage{amsthm}
\usepackage{enumerate}
\usepackage{graphicx}

\newtheorem{thm}{Theorem}[section]
\newtheorem{theorem}[thm]{Theorem}
\newtheorem{lemma}[thm]{Lemma}

\newtheorem{corollary}[thm]{Corollary}

\renewcommand{\a}{\alpha}
\renewcommand{\l}{\lambda}
\newcommand{\lmin}{\lambda_{\rm min}}

\newcommand{\RR}{\mathbb{R}}

\title{Fractional decompositions and the smallest-eigenvalue separation}

\author{Fiachra Knox\thanks{Supported by a PIMS Postdoctoral Fellowship.}\\
      Department of Mathematics\\
      Simon Fraser University\\
      \texttt{fiachraknox@hotmail.com}
  \and
    Bojan Mohar%
      \thanks{B.M.~was supported in part by the NSERC Discovery Grant R611450 (Canada), by the Canada Research Chairs program, and by the Research Project J1-8130 of ARRS (Slovenia).}~%
      \thanks{On leave from IMFM, Department of Mathematics, University of Ljubljana.}\\
      Department of Mathematics\\
      Simon Fraser University\\
      \texttt{mohar@sfu.ca}
}

\date{\today}

\begin{document}

\maketitle

\begin{abstract}
A new method is introduced for bounding the separation between the value of $-k$ and the smallest eigenvalue of a non-bipartite $k$-regular graph. The method is based on fractional decompositions of graphs. As a consequence we obtain a very short proof of a generalization and strengthening of a recent result of Qiao, Jing, and Koolen \cite{QJK19} about the smallest eigenvalue of non-bipartite distance-regular graphs.
\end{abstract}

\section{Introduction}

Let us consider a connected graph of order $n$ and its adjacency matrix $A(G)$. We speak of the \emph{eigenvalues} the graph $G$, by which we mean the eigenvalues of $A(G)$.  We order the eigenvalues in descending order, $\l_1(G) \ge \l_2(G) \ge \cdots \ge \l_n(G)$. By the Perron-Frobenius Theorem, the \emph{spectral radius} of $A(G)$ is equal to $\l_1(G)$. This means that $|\l_n(G)|\le \l_1(G)$. It also follows by the Perron-Frobenius Theorem that $|\l_n(G)| = \l_1(G)$ if and only if the graph $G$ is bipartite. For convenience, we write $\lmin(G) = \l_n(G)$. In this note we are interested primarily in the \emph{gap}
$$
     \delta(G) = \lmin(G) + \l_1(G) \ge 0,
$$
which in some sense measures how far are we from being bipartite. We refer to \cite{CDS}, and in particular to the famous inequality of Hoffman \cite{Hoffman}, see also \cite{Bilu}.

If $G$ is $k$-regular, then $\l_1(G)=k$ and $\delta(G)=\lmin(G)+k$. Our main concern will be to derive bounds on $\delta(G)$.

Our original motivation for considering $\delta(G)$ was the following recent result of Qiao, Jing, and Koolen \cite{QJK19} about the smallest eigenvalue of nonbipartite distance-regular graphs.

\begin{theorem}[Qiao, Jing, and Koolen \cite{QJK19}]
\label{thm:QJK}
Let $G$ be a non-bipartite distance-regular graph with valency $k$ and odd girth $g$. Then there exists a constant $\varepsilon(g) > 0$ such that $\delta(G) \ge \varepsilon(g)k$.
\end{theorem}

The proof in \cite{QJK19} gives a stronger result that for every odd integer $g\ge3$ there exists $\varepsilon(g)>0$ such that every non-bipartite distance-regular graph $G$ of odd girth $g$ has $\delta(G)\ge \varepsilon(g)k$, where $k$ is the degree of the vertices in $G$. We will mean this version when referring to Theorem \ref{thm:QJK}.

This note yields a very short proof and a strengthening of this results (Corollary \ref{cor:DRG})and gives a generalization to arbitrary graphs (Theorem~\ref{thm:min eigenvalue}).

Theorem \ref{thm:QJK} was used in \cite{QJK19} to classify all non-bipartite distance-regular graphs of diameters $D=4$ and $D=5$ that have $\delta(G)\le k/D$, continuing on their earlier result in \cite{QK18} which was used for diameter $D=3$ and was based on a similar spectral lemma bounding $\delta(G)$.

It is also mentioned in \cite[Remark 1.2]{QJK19} that $\varepsilon(g) \to 0$ as $g \to \infty$, and odd cycles show that $\varepsilon(g) \le 2 \cos^2(\tfrac{(g-1)\pi}{2g})$ (for $k=2$).

This result is interesting because it bounds the smallest eigenvalue away from the value $-k$ which is attained only for bipartite graphs. However, the dependence of the bound on $k$ is not tight. Our note gives the following improvement, which is asymptotically best possible and holds not only for distance-regular graphs but holds for arbitrary $k$-regular graphs, as long as they possess certain homogeneity property. Our main result, Theorem \ref{thm:min eigenvalue}, is given in the next section; its application to distance-regular graphs is outlined in Section \ref{sect:DRG}.

\section{Fractional decompositions and smallest eigenvalue of a graph}

A \emph{decomposition} of a graph $G$ is a partition of its edge-set into subsets $E_1,\dots,E_r$. Each subset $E_i$ determines a subgraph $G_i=(V(G),E_i)$ of $G$. We will denote by $A_i$ the adjacency matrix of $G_i$. Then we write $G=\sum_{i=1}^r G_i$ and we have $A(G)=\sum_{i=1}^r A_i$.

A \emph{fractional decomposition} of $G$ with nonnegative weights $\a_i\ge0$ ($i=1,\dots,r$) is a collection of spanning subgraphs $G_i=(V(G),E_i)$ ($i=1,\dots,r$), whose edge-sets $E_i$ are not necessarily edge-disjoint, such that for each edge $e\in E(G)$,
\begin{equation}
\label{eq:edge condition}
   \sum_{i:e\in E_i} \a_i = 1.
\end{equation}
Then we write $G=\sum_{i=1}^r \a_i G_i$ and we have $A(G)=\sum_{i=1}^r \a_i A_i$. We request that the subgraphs $G_i$ are spanning in order to have the latter correspondence between their adjacency matrices and $A(G)$. However, we will consider their subgraphs $G_i'$ that are obtained from $G_i$ by removing all isolated vertices (i.e., vertices of degree 0). We also let $V_i=V(G_i')$ be the set of vertices of positive degree in $G_i$.

The fractional decomposition is \emph{homogeneous} if there are regular graphs $H_1,\dots,H_t$ such that each $G_i'$ is isomorphic to one of these graphs and for each $j\in [t]$, the value
\begin{equation}
\label{eq:homogeneous}
   s_j = s_j(v) = \sum_{i: G_i \approx H_j, v\in V_i} \a_i
\end{equation}
is the same for each vertex $v\in V(G)$.

Figure \ref{fig:1} shows two graphs that have homogeneous fractional decomposition into copies of $H_1=C_3$ and $H_2=C_5$ with weights $\tfrac{1}{2}$ for every subgraph, where the cycles in the decomposition are all facial cycles. (The faces of the first one correspond to the planar embedding, including the outer face, and the faces for the other one correspond to the faces of the embedding in the projective plane, where diametrically opposite points and edges on the circle are pairwise identified. In the second example one can also take weight $\a$ ($0\le\a \le 1$) for the triangles and weight $1-\a$ for the 5-cycles.

\begin{figure}[htb]
\centering
\includegraphics[width=0.9\linewidth]{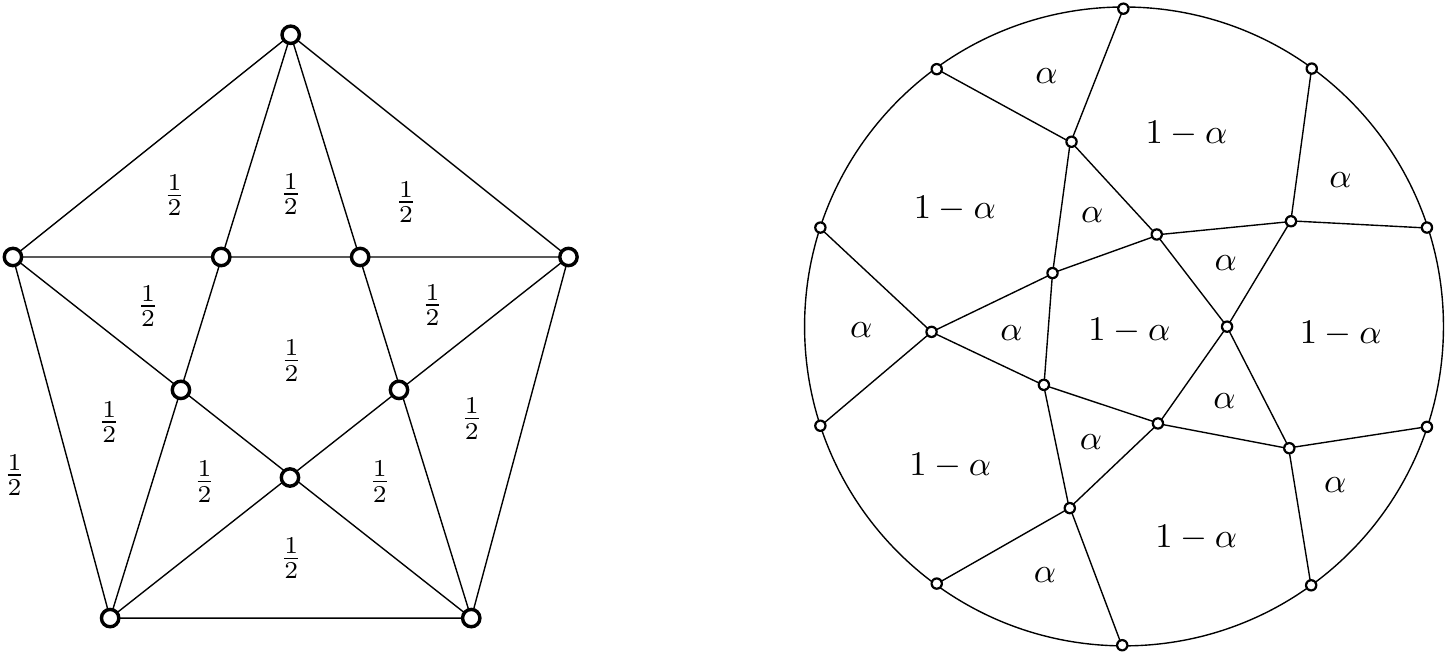}
\caption{Two examples of homogeneous fractional decompositions in which the subgraphs of the decomposition correspond to the facial cycles of lengths 3 and 5 of the planar and projective-planar embeddings shown.}
\label{fig:1}
\end{figure}

\begin{lemma}
\label{lem:homogeneous fractional partition count}
Suppose that $G_1,\dots,G_r$ is a homogeneous fractional partition of a $k$-regular graph $G$ with weights $\a_1,\dots,\a_r$. Let $H_1,\dots,H_t$ be the corresponding graphs, where $H_j$ is $d_j$-regular, $j\in [t]$, and the constants $s_j$ are given by $($\ref{eq:homogeneous}$)$. Then we have
\begin{equation}
\label{eq:lemma}
   \sum_{j=1}^t d_js_j = k.
\end{equation}
\end{lemma}

\begin{proof}
The sum of $d_js_j$ is the sum of the weights of edges in $G_i$ on all edges of $G$ incident with a vertex $v$. By the edge-condition (\ref{eq:edge condition}), this sum contributes 1 to each edge incident with $v$. Since $G$ is $k$-regular, the sum must be equal to $k$.
\end{proof}

\begin{theorem}
\label{thm:min eigenvalue}
Suppose that $G_1,\dots,G_r$ is a homogeneous fractional partition of a $k$-regular graph $G$ with weights $\a_1,\dots,\a_r$. Let $H_1,\dots,H_t$ be the corresponding regular graphs, and let the constants $s_j$ $(j\in [t])$ be given by {\rm (\ref{eq:homogeneous})}. Then we have
$$
   \delta(G) \ge \sum_{j=1}^t \delta(H_j) s_j.
$$
\end{theorem}

\begin{proof}
For $i\in [r]$, let $V_i$ be the set of vertices of $G_i'$.
Let $x\in \RR^{V(G)}$ be a unit eigenvector for $\lmin(G)$. For each $i\in [r]$, let $x^i$ be the vector obtained from $x$ by changing the coordinates $x_v$ to 0 if $v$ has degree 0 in $G_i$. Then
\begin{equation}
\label{eq:shorter vector}
   (A_ix,x) = (A_i x^i, x^i) \ge \lmin(A_i)(x^i,x^i).
\end{equation}
Therefore,
\begin{eqnarray*}
\label{eq:sum}
   \lmin(G) &=& (Ax,x) = \Bigl(\sum_i \a_iA_ix,x\Bigr) = \sum_i \a_i (A_i x^i, x^i) \\
   &\ge& \sum_{i=1}^r \a_i \, \lmin(A_i)(x^i,x^i) \\
   &=& \sum_{j=1}^t \sum_{i: G_i' \approx H_j} \a_i\,\lmin(H_j) \sum_{v\in V_i} x_v^2 \\
   &=& \sum_{j=1}^t \lmin(H_j) \sum_{v\in V(G)} x_v^2 \sum_{i: G_i' \approx H_j, v\in V_i} \a_i \\
   &=& \sum_{j=1}^t \lmin(H_j)\, s_j \sum_{v\in V(G)} x_v^2 = \sum_{j=1}^t \lmin(H_j)\, s_j \\
   &=& \sum_{j=1}^t (\delta(H_j) - d_j) s_j = -k + \sum_{j=1}^t \delta(H_j) s_j.
\end{eqnarray*}
This implies that $\delta(G) \ge \sum_{j=1}^t \delta(H_j) s_j$, which we were to prove.
\end{proof}

Let us illustrate Theorem \ref{thm:min eigenvalue} on some examples.

\begin{itemize}
\item[(a)]
Consider the first graph depicted in Figure \ref{fig:1}. Here $H_1=C_3$ and $H_2=C_5$, the fractional decomposition has all weights equal to $\tfrac{1}{2}$ and has 10 copies of $C_3$ and 6 copies of $C_5$. Since each vertex is in three copies of $C_3$ and one $C_5$, we have $s_1=\tfrac{3}{2}$ and $s_2=\tfrac{1}{2}$. This gives that
$$
   \delta(G) \ge \tfrac{3}{2} \delta(C_3) + \tfrac{1}{2} \delta(C_5) =
   \tfrac{3}{2}(-1+2) + \tfrac{1}{2}(\tfrac{-1-\sqrt{5}}{2} +2) =
   \tfrac{9-\sqrt{5}}{4}.
$$
\item[(b)]
Similarly as above we get for the second example in Figure \ref{fig:1} that
$$ 
   \delta(G) \ge 2\a + 2(1-\a)(\tfrac{3-\sqrt{5}}{2}) = 3-\sqrt{5} + \alpha(\sqrt{5}-1).
$$
This bound is strongest for $\a=1$, and shows that $\delta(G)\ge2$.
\item[(c)]
The graph in (b) is the line graph of the Petersen graph. In fact, every line graph $G=L(H)$ of a $k$-regular graph ($k\ge3$) has a homogeneous decomposition into cliques of order $k$, two cliques per vertex. Therefore, $\delta(G) \ge 2\delta(K_k) = 2(k-2)$.
Of course, this is equivalent to the well-known fact that the smallest eigenvalue of a line graph $L(H)$ is $-2$ if $|E(H)|>|V(H)|$. So, this fact is not new, but these examples show that our theorem is tight for every even degree. 
\item[(d)]
Another class of tight examples can be obtained as follows.
By taking the ``blow-up" of the odd cycle $C_{2h+1}$ in which we replace each vertex $v$ by an independent set $I_v$ of cardinality $k$ and each edge $uv$ by a copy of the complete bipartite graph $K_{k,k}$ joining $I_v$ and $I_u$, we obtain a graph of degree $2k$ and odd girth $2h+1$ which has a homogeneous fractional decomposition into copies of the cycle $C_{2h+1}$. The bound of Theorem \ref{thm:min eigenvalue} is tight for all such graphs.
\end{itemize}

The setup of homogeneous fractional decompositions can be applied in other similar settings if the subgraphs $G_i'$ are spanning (i.e. for each $i$, every vertex of $G$ is incident with an edge in $G_i$). In that case, bounds similar to that of Theorem \ref{thm:min eigenvalue} can be derived for the \emph{spectral gap} $\l_1(G)-\l_2(G)$, or for any convex function of eigenvalues, like the sum of $t$ largest or $t$ smallest eigenvales. 

Note that the \emph{degree matrix} $D(G)$ and the \emph{Laplacian matrix} $L(G)=D(G)-A(G)$ are also fractionally decomposed with the same coefficients, $D(G) = \sum_{i=1}^r \a_i D(G_i)$ and $L(G) = \sum_{i=1}^r \a_i L(G_i)$. If each $G_i'$ is spanning, then there is a result similar to Theorem \ref{thm:min eigenvalue} bounding the largest eigenvalue of $L(G)$ and also bounding the smallest non-trivial eigenvalue of $L(G)$.

\section{Distance-regular graphs}
\label{sect:DRG}

The following decomposition lemma enables us to apply the bound of Theorem \ref{thm:min eigenvalue} to distance-regular graphs.

\begin{lemma}
Let $G$ be a distance-regular graph with odd girth $g=2h+1$. Then $G$ admits a homogeneous fractional decomposition into copies of the odd cycle $C_{2h+1}$.
\end{lemma}

\begin{proof}
With the standard notation for intersection array for distance-regular graphs, it is clear that for every pair of vertices $v$ and $x$ at distance $h$, there are precisely $p := c_hc_{h-1}\cdots c_1$ paths of length $h$ from $x$ to $v$. Also, distance-regularity implies that for every edge $xy\in E(G)$, the number $q$ of vertices that are at distance $h$ from both of them is independent of the edge. Since the odd girth is $2h+1$, we know that $q>0$. This implies that every edge belongs to precisely $p^2q$ cycles of length $2h+1$. (To see this one has to realize that a path of length $h$ from $x$ to $v$ and such path from $y$ to $v$ cannot intersect, since that would give a shorter odd cycle; hence any two such paths together with the edge form a cycle.) Let $G_1,\dots,G_r$ be all cycles of length $2h+1$ in $G$. Then it is clear that $G = \sum_{i=1}^r \tfrac{1}{p^2q} \, G_i$.
\end{proof}

\begin{corollary}
\label{cor:DRG}
Let $G$ be a $k$-regular non-bipartite distance-regular graph with odd girth $g=2h+1$. Then
$$
  \delta(G) \ge (1-\cos(\tfrac{\pi}{2h+1}))\,k > 
  \biggl( \frac{\pi^2}{2(2h+1)^2} - \frac{\pi^4}{24(2h+1)^4} \biggr)\, k \,.$$
\end{corollary}

\begin{proof}
By the previous lemma, $G$ has a homogeneous fractional decomposition into copies of the graph $H=C_{2h+1}$. It is well-known that
$$
   \lmin(C_{2h+1}) = 2 \cos(\tfrac{2h\,\pi}{2h+1}) =
   - 2 \cos(\tfrac{\pi}{2h+1}) > -2 + (\tfrac{\pi}{2h+1})^2 - \tfrac{1}{12}(\tfrac{\pi}{2h+1})^4.
$$
Now Theorem \ref{thm:min eigenvalue} completes the proof.
\end{proof}


\begin{thebibliography}{9}

\bibitem{Bilu}
Yonatan Bilu, Tales of {H}offman: three extensions of {H}offman's bound on the graph chromatic number,
J. Combin. Theory Ser. B 96 (2006) 608--613.

\bibitem{CDS}
D.M. Cvetković, M. Doob, H. Sachs,
Spectra of graphs. Theory and applications (Third edition),
Johann Ambrosius Barth, Heidelberg, 1995.

\bibitem{Hoffman}
A.J. Hoffman, On eigenvalues and colorings of graphs,
Graph Theory and Its Applications, Proc. Adv. Sem., Math. Research Center, Univ. of Wisconsin, Madison, WI, 1969, Academic Press, New York (1970), pp. 79--91.

\bibitem{QK18}
Z. Qiao and J. Koolen, A valency bound for distance-regular graphs, J. Combin. Theory, Ser. A, 155:304--320, 2018.

\bibitem{QJK19}
Zhi Qiao, Yifan Jing, and Jack Koolen,
Non-bipartite distance-regular graphs with a small smallest eigenvalue,
Electronic J. Combin. 26(2) (2019), \#P2.41.

\end{thebibliography}

\end{document}